\documentclass[10pt,a4]{amsart}

\usepackage[mathscr]{eucal}
\usepackage{amssymb}
\usepackage{latexsym}
\usepackage{amsthm}
\usepackage{amscd}
\usepackage{amsmath}

\theoremstyle{plain}
\newtheorem{theorem}{Theorem}[section]
\newtheorem{proposition}[theorem]{Proposition}
\newtheorem{corollary}[theorem]{Corollary}

\theoremstyle{definition}

\newtheorem{example}[theorem]{Example}

\def\maru#1{{\ooalign{\hfill$\scriptstyle#1$\hfill\crcr$\bigcirc$}}}
\newcommand{\SF}{\mathscr{F}}
\title{Compound basis arising from the basic $A^{(1)}_{1}$-module }
\author{Kazuya Aokage,
 Hiroshi Mizukawa 
and Hiro-Fumi Yamada
}
\date{}
\address{Kazuya Aokage, Department of Mathematics, Okayama University, Okayama 700-8530, 
Japan}
\email{k.aokage@math.okayama-u.ac.jp}
\address{Hiroshi Mizukawa, Department of Mathematics,  National Defense
Academy in Japan,  Yokosuka 239-8686, Japan}
\email{mzh@nda.ac.jp}
\address{Hiro-Fumi Yamada, Department of Mathematics, Okayama University, Okayama 700-8530, 
Japan}
\email{yamada@math.okayama-u.ac.jp}
\date{}

\begin{document}

\begin{abstract}  
A new basis for the polynomial ring of infinitely many variables
is constructed which consists of products of Schur functions and $Q$-functions.
The transition matrix from the natural Schur function basis is investigated.
 \end{abstract}
\maketitle
\section{Introduction}
This note concerns with realizations of the basic representation of the affine Lie algebra of type
$A_{1}^{(1)}$(cf. \cite{kac}). The most well-known realization is $PU$, principal, untwisted, whose representation space is 
$$\SF^{PU}={\mathbb C}[t_{j};j \geq 1, {\rm odd}].$$ 
In the context of nonlinear integrable systems, this space appears as that of the $\mathrm{KdV}$ hierarchy.  
The second one is $HU$, homogeneous, untwisted, which is on 
$$\SF^{HU}=\bigoplus_{m \in {\mathbb Z}}\SF(m);\ \SF(m)={\mathbb C}[t_{j};j \geq 1]\otimes q^m.$$  
This space is for the $\mathrm{NLS}$ (nonlinear Schr\"{o}dinger) hierarchy
and also for the Fock representation of the Virasoro algebra(cf. \cite{iy}). 
The third one is $PT$, principal, twisted, on $\SF^{PT}$ which coincides with $\SF^{PU}$.
And the fourth one is $HT$, homogenous, twisted, on $\SF^{HT}$ which is the same as $\SF^{HU}$. 
The Lie algebra of type $A_{1}^{(1)}$ is isomorphic to that of type $D_{2}^{(2)}$.
One can discuss twisted realization of $A_{1}^{(1)}$-modules via this isomorphism.

The purpose of this note is to give a weight basis for $\SF^{HT}$ and compare it with a standard Schur function basis for $\SF^{HU}$. We will show that the transition matrix has several interesting combinatorial 
features. This is a detailed version of our announcement \cite{amy}.

\section{A quick review of realizations}
Let us first consider the principal untwisted realization on 
$\SF^{PU}={\mathbb C}[t_{j};j\geq 1, {\rm odd}]$.
To describe a weight basis for this space we need Schur functions and 
Schur's $Q$-functions in our setting. Let ${ P}_n$ be the set of all partitions of $n$ and put
$P=\bigcup_{n \geq 0} P_{n}$.
For $\lambda \in P_{n}$, the Schur function $S_{\lambda}(t)$ is defined by 
$$S_{\lambda}(t)=\sum_{\rho=(1^{m_{1}}2^{m_{2}}\ldots) \in P_{n}}\chi_{\rho}^{\lambda}\frac{t_{1}^{m_{1}}
t_{2}^{m_{2}}\cdots}{m_{1}!m_{2}!\cdots} ,$$
where the summation runs over all partitions $\rho=(1^{m_{1}}2^{m_{2}}\ldots)$ of $n$, 
and $\chi_{\rho}^{\lambda}$ is the irreducible character of the symmetric group ${\mathfrak S}_{n}$,
indexed by $\lambda$ and evaluated at the conjugacy class $\rho$. The Schur functions are the ordinary irreducible characters of the general linear groups. If the group element $g$ has eigenvalues $x_1,x_2, \ldots$, then the original irreducible character is recovered by putting $p_j:=jt_j$ $(j\geq 1)$, where $p_j=\sum_{i\geq 1}x_{i}^j$ is the $j$-th power sum of the eigenvalues.

The 2-reduction of a polynomial $f(t)$ is to ``kill" the even numbered variables $t_{2}, t_{4}, \ldots$,
i.e. , 
$$f^{(2)}(t)=f(t)|_{t_{2}= t_{4}=\ldots=0} \in \SF^{PU}.$$ 
The 2-reduced Schur functions are linearly dependent in general.
However all linear relations among  them  are known, and one can 
choose certain set  $P' \subset P$ so that
$\left\{ S_{\lambda}^{(2)} ; \lambda \in P'  \right\}$
forms a basis for $\SF^{PU}$ (cf. \cite{any}).

The space $\SF^{PU}$ also affords the principal twisted realization.
A weight basis is best described by Schur's $Q$-functions.
Let $SP_{n}$ (resp. $OP_{n}$) be the set of all strict (resp. odd)  partitions  
of $n$ and put $SP=\bigcup_{n \geq 0}SP_{n}$,
$OP=\bigcup_{n \geq 0}OP_{n}$.
For $\lambda \in SP_{n}$, the $Q$-function $Q_{\lambda}(t)$ is defined by 
$$Q_{\lambda}(t)=
\sum_{\rho=(1^{m_{1}}3^{m_{3}}\ldots) \in OP_{n}}
2^{\frac{\ell(\lambda)-\ell(\rho)+\epsilon}{2}}
\zeta_{\rho}^{\lambda}\frac{t_{1}^{m_{1}}
t_{3}^{m_{3}}\cdots}{m_{1}!m_{3}!\cdots} ,$$
where the summation runs over all odd partitions $\rho=(1^{m_{1}}3^{m_{3}}\ldots)$
of $n$, $\epsilon =0$ or $1$ according to that $\ell(\lambda)-\ell(\rho)$
is even or odd and $\zeta_{\rho}^{\lambda}$ is the irreducible spin 
character of  ${\mathfrak S}_{n}$,
indexed by $\lambda$ and evaluated at the conjugacy class $\rho$. For the $Q$-functions, we set $p_j:=\frac{1}{2}jt_j$ $(j\geq 1,\mathrm{odd})$ as the relation with the ``eigenvalues".
A more detailed account is found in \cite{ny}.
Here we remark the relation of $Q$-functions and the $P$-functions. We define inner product $\langle\  ,\ \rangle_{q}$ on $\SF(0)$ by $\langle p_{\lambda},p_{\mu}\rangle_{q}=z_{\lambda}(q)\delta_{\lambda \mu}$, where $z_{\lambda}(q)=z_{\lambda}\prod_{i \geq 1}(1-q^{\lambda_{i}})^{-1}$. Note that $z_{\lambda}(-1)$ cannot be defined for $\lambda$ which has even parts. Therefore we have to re-define $\langle\  ,\ \rangle_{-1}$ by setting  $\langle p_{\lambda},p_{\mu}\rangle_{-1}=2^{-\ell(\lambda)}z_{\lambda}\delta_{\lambda \mu}$. 
The $P$-functions are dual to the $Q$-functions with respect to the inner product $\langle\ ,\ \rangle_{-1}$ on $\SF^{PU}.$ For a strict partition $\lambda$, we see that $P_{\lambda}(t)=2^{-\ell(\lambda)}Q_{\lambda}(t)$ (cf. \cite{mac}).


In order to give the homogeneous, twisted realization we employ 
a combinatorics of strict partitions.
We introduce the following h-abacus.
For example, the h-abacus of   $\lambda=(11,10,5,3,2)$ is shown below. 
\begin{align*}
 {\begin{array}{ccc}
&{1}&\maru{3}\\
\maru{2}&&\\
4&\maru5&7\\
{6}&&\\
8&{9}&\maru{11}\\
\maru{10}&&\\
12&13&15\\
\vdots&\vdots&\vdots
\end{array}}
\end{align*}

From this h-abacus of $\lambda$ we read off a triplet  $(\lambda^{hc};\lambda^h[0],
\lambda^h[1])$ of partitions. Firstly $\lambda^h[0]=(5,1)$, a strict partition obtained just by 
taking halves of the circled positions of the leftmost column. 

For obtaining $\lambda^h[1]$, we need the following process:
\begin{enumerate}
\item
For the third column, the circled positions correspond to the vacancies "$\circ$".   
\item
For the second column, the circled positions correspond to being occupied "$\bullet$".  
\item
Read the third column from infinity to the position 3 and consequently the second 
column from the position $1$ to infinity, and draw the Maya diagram
$$
\begin{array}{ccccccccc}
\dots&15&11&7&3&1&5&9&\dots.\\
&\underline{\bullet}&\circ&\bullet&\circ&\circ&\bullet&\underline{\circ}&
\end{array}
$$
\item
For each $\bullet$, count the number of vacancies 
which are on the left of that $\bullet$, and get a partition 
$$\lambda^h[1]=(3,1).$$ 
\end{enumerate}

\ \ \ Next the h-core $\lambda^{hc}$ is obtained by the following moving and removing:
\begin{enumerate}
\item Remove all circles on the leftmost column.
\item Move a circle one position up along the second or the third column.
\item Remove the two circles at the positions 1 and 3 simultaneously.
\item The ``stalemate" determines the partition 
$$\lambda^{hc}=(3).$$
 \end{enumerate}

Note that $\lambda^{hc}$ is always of the form 
$$\Delta^h(m)=(4m-3,4m-7,\ldots,5,1){\ \rm or\ }\Delta^h(-m)=(4m-1,4m-5,\ldots,7,3)$$
for some $m \in {\mathbb N}\ (\Delta^h(0)=\emptyset )$. Let $HC$ be the set of 
all such $\lambda^{hc}$'s.
In this way we have a one-to-one correspondence between $\lambda \in SP$
and $(\lambda^{hc}; \lambda^h[0],\lambda^h[1]) \in HC \times SP \times P$  
with the condition 
$$|\lambda|=|\lambda^{hc}|+2(|\lambda^h[0]|+2|\lambda^h[1]|).$$
By making use of this one-to-one correspondence, we define the  
linear map $\eta:\SF^{PT}\rightarrow \SF^{HT}$ by
$$\eta(Q_{\lambda}(t))=Q_{\lambda^h[0]}(t)S_{\lambda^h[1]}(t')\otimes q^{m(\lambda)}. $$
Here $$m(\lambda)=({\rm number\ of\ circles\ on\ the\ second\ column} )-
({\rm number\ of\ circles\ on\ the\ third\ column} )$$ 
and $S_{\nu}(t')=S_{\nu}(t)|_{t_{j}\mapsto t_{2j}}$ for any $j \geq 1$.
For any integer $m$, the set 
$$\{\eta(Q_{\lambda});\lambda \in SP, m(\lambda)=m\}$$ 
forms a basis for $\SF(m)={\mathbb C}[t_{j}; j \geq 1]\otimes q^m $ (cf. \cite{imny}).
  Under the condition $m=0$, there is a one-to-one correspondence
  between  the following two sets for any 
  $n \geq 0$:
  \begin{enumerate}
  \item[(i)] $\{ \lambda \in SP_{2n}; \lambda^{hc}=\emptyset \},$
  \item[(ii)] $\{(\mu,\nu) \in SP_{n_{0}}\times P_{n_{1}}; n_{0}+2n_{1}=n\}.$ 
\end{enumerate}

\section{Compound basis}

We begin with some bijections between sets of partitions.
The first one is 
 $$\phi:P_{n} \longrightarrow  \bigcup_{n_{0}+2n_{1}=n}SP_{n_{0}}\times P_{n_{1}}$$
defined by
 $\lambda \mapsto (\lambda^{r},\lambda^{d}).$
Here the multiplicities $m_{i}(\lambda^{r})$ and
$m_{i}(\lambda^{d})$ of $i\geq 1$
are given respectively by
 \begin{align*}
  m_{i}(\lambda^{r})=
 \begin{cases}
 1  & m_{i}(\lambda) \equiv 1 \pmod 2\\
 0  & m_{i}(\lambda) \equiv 0 \pmod 2,
 \end{cases}
 \end{align*}
 and
  \begin{align*}
  m_{i}(\lambda^{d})=
 \begin{cases}
 \frac{1}{2}(m_{i}(\lambda)-1)  & m_{i}(\lambda) \equiv 1 \pmod 2\\
  \frac{1}{2}(m_{i}(\lambda))   & m_{i}(\lambda) \equiv 0 \pmod 2.
 \end{cases}
 \end{align*}
 For example, 
 if $\lambda=(5^34^42^71)$, then $\lambda^{r}=(521)$ and $\lambda^{d}=(54^22^3)$.
We set
 $$P_{n_{0},n_{1}}=\phi^{-1}(SP_{n_{0}}\times P_{n_{1}}).$$

The second bijection is 
$$\psi:P_{n} \longrightarrow \bigcup_{n_{1}+2n_{2}=n}OP_{n_{1}}\times P_{n_{2}}$$
defined by
$\psi(\lambda)=(\lambda^o,\lambda^e)$.
Here
 $\lambda^{o}$ is obtained by 
 picking up the odd parts of $\lambda$,
 while $\lambda^e$ is obtained by 
 taking halves of the even parts. For example,  
if  $\lambda=(5^34^42^71),$ then $\lambda^o=(5^31)$ and $\lambda^e=(2^41^7)$.

The third bijection is called the Glaisher map.
Let 
$\lambda=(\lambda_{1},\lambda_{2},\ldots)$
be a strict partition of $n$. Suppose that
$\lambda_{i}=2^{p_{i}}q_{i}\ (i=1,2,\ldots)$, where $q_{i}$ is odd.
Then an odd partition $\tilde{\lambda}$ of $n$ is defined by
$$m_{2j-1}(\tilde{\lambda})=\sum_{q_{i}=2j-1,i \geq 1}2^{p_{i}}.$$
For example, if 
$\lambda=(8,6,4,3,1),$ then $\tilde{\lambda}=(3^3,1^{13})$.
This gives a bijection between $SP_{n}$ and $OP_{n}$. 
\begin{proposition}
Let $(n_{0},n_{1})$ be fixed. Then we have
\begin{align*}
\sum_{\lambda \in P_{n}}\ell(\lambda)
 &=\sum_{\lambda \in P_{n}}(\ell(\lambda^{r})+2\ell(\lambda^{d}))
 =\sum_{\lambda \in P_{n}}(\ell(\lambda^o)+\ell(\lambda^e))
 =\sum_{\lambda \in P_{n}}(\ell(\tilde{\lambda^{r}})+\ell(\lambda^e)),\\
 \sum_{\lambda \in P_{n_{0},n_{1}}}\ell(\lambda)
 &=\sum_{\lambda \in P_{n_{0},n_{1}}}(\ell(\lambda^{r})+2\ell(\lambda^{d}))
 =\sum_{\lambda \in P_{n_{0},n_{1}}}(\ell(\lambda^o)+\ell(\lambda^e)),\\
\sum_{\lambda \in P_{n}}2\ell(\lambda^{d})
&=\sum_{\lambda \in P_{n}}2\ell({\lambda^e})
=\sum_{\lambda \in P_{n}}(\ell(\lambda^o)+\ell(\lambda^e)-\ell(\lambda^{r}))
=\sum_{\lambda \in P_{n}}(\ell(\tilde{\lambda^{r}})+\ell(\lambda^e)-\ell(\lambda^{r})), \\ and\\
\sum_{\lambda \in P_{n_{0},n_{1}}}2\ell(\lambda^{d})
&=\sum_{\lambda \in P_{n_{0},n_{1}}}(\ell(\lambda^o)+\ell(\lambda^e)-\ell(\lambda^{r})).
\end{align*}
\end{proposition}
Looking at the representation spaces $\SF^{HU}$ and $\SF^{HT}$,
we have the following two natural bases for the space 
$$\SF(0)_{n}={\mathbb C}[t_{j}; j \geq 1]_{n}$$
consisting of the homogenous polynomials of degree $n$ subject to deg $t_{j}=j$.
Namely we have
\begin{enumerate}
\item[(i)] $\{S_{\lambda}(t); \lambda \in P_{n}\}$,
\item[(ii)] $\{Q_{\lambda^{r}}(t)S_{\lambda^{d}}(t'); \lambda \in P_{n}\}$.
\end{enumerate}
For simplicity we write
$$W_{\lambda}(t)=Q_{\lambda^{r}}(t)S_{\lambda^{d}}(t')$$
for $\lambda \in P_{n}$ and call the set (ii) the compound basis for  $\SF(0)_{n}$.

Our problem is to determine the transition matrix between 
these two bases. 
Let $A_{n}=(a_{\lambda \mu})$ be defined by
\begin{equation}
S_{\lambda}(t)=\sum_{\mu \in P_n}a_{\lambda \mu}W_{\mu}(t)
\end{equation}
for $\lambda \in P_{n}$.
\\\ \ Here we remark the relation between our basis and the $Q^{'}$-functions. Lascoux, Leclerc and Thibon (cf. \cite{LLT}) introduced the $Q^{'}$-functions as the basis for $\SF(0)_{n}$ dual to $P$-functions with respect to the inner product
$$\langle F(t),G(t)\rangle_{0}:=F(\tilde{\partial})\overline{G(t)}|_{t=0},$$ where $\tilde{\partial}=(\frac{\partial}{\partial t_1}, \frac{1}{2}\frac{\partial}{\partial t_2}, \frac{1}{3}\frac{\partial}{\partial t_3}, \ldots)$. For a strict partition $\mu$ we see that $Q_{\mu}^{'}(t)=Q_{\mu}(2t)$. For a partition $\lambda$ which is not necessarily strict, we see that $$Q_{\lambda}^{'}(t)=Q_{\lambda^{r}}(2t)h_{\lambda^{d}}(t^{'})$$ where $h_{\lambda^{d}}$ is the complete symmetric function indexed by $\lambda^{d}$. Therefore the transition from $W_{\lambda}$ to $Q_{\mu}^{'}$ is essentially given by the Kostka numbers.

\section{transition matrices}
In the previous section, functions are expressed in terms of the ``time variables" $t=(t_1,t_2,\ldots)$ of the soliton equations. However, for the description and the proof of our formula, it is more convenient to use the ``original" variables of the symmetric functions, i.e., the eigenvalues $x=(x_1, x_2, \ldots).$ 
\\\ \ The definition $(1)$ of $a_{\lambda \mu}$ is rewritten as 
 $$S_{\lambda}(x,x)=\sum_{\mu \in P_{n}}a_{\lambda \mu}Q_{\mu^r}(x)S_{\mu^d}(x^2),$$where $(x,x)=(x_1,x_1,x_2,x_2,\ldots)$ and $x^2=(x_{1}^2,x_{2}^2,\ldots).$
 Hereafter we will denote 
$$ W_{\lambda}(x)=Q_{\lambda^{r}}(x)S_{\lambda^{d}}(x^2),\ V_{\lambda}(x)=P_{\lambda^{r}}(x)S_{\lambda^{d}}(x^2).
 $$
 Also we set the following spaces of symmetric functions
$$\Lambda={\mathbb C}[p_{r}(x); r \geq 1],\ 
\Gamma={\mathbb C}[p_{r}(x); r \geq 1,\rm{odd}],$$ 
and $$\Gamma^{'}={\mathbb C}[p_{r}(x); r \geq2, \rm{even}]$$
so that $$\Lambda \cong \Gamma \otimes \Gamma^{'}. $$ We have two bases for $\Lambda$:
$$W=(W_{\lambda}(x))_{\lambda}\ {\rm and} \ V=(V_{\lambda}(x))_{\lambda}.$$
First we notice the following Cauchy identity.
%
\begin{proposition}
$$\prod_{i,j \geq 1}\frac{1}{(1-x_{i}y_{j})^2}=\sum_{\lambda \in P}W_\lambda(x)V_{\lambda}(y)
. $$
\end{proposition}
\begin{proof}
We compute
\begin{align*}
\sum_{\lambda \in P}W_\lambda(x)V_{\lambda}(y)&=
\sum_{\lambda \in P}
Q_{\lambda^{r}}(x)S_{\lambda^{d}}(x^2)P_{\lambda^{r}}(y)S_{\lambda^{d}}(y^2)\\
&=
\sum_{\mu \in SP}Q_{\mu}(x)P_{\mu}(y)
\sum_{\nu \in P}S_{\nu}(x^2)S_{\nu}(y^2).\\
\end{align*}
Taking the inner products $\langle\  ,\ \rangle_{-1}$ and $\langle\  ,\ \rangle_{0}$ on $\Lambda$, we obtain $$\sum_{\mu \in SP}Q_{\mu}(x)P_{\mu}(y)=\prod_{i,j}\frac{1+x_{i}y_{j}}{1-x_{i}y_{j}}, $$ and $$\sum_{\nu \in P}S_{\nu}(x^2)S_{\nu}(y^2)=\prod_{i,j}\frac{1}{1-x_{i}^{2}y_{j}^{2}}.$$ We have $$\sum_{\lambda \in P}W_\lambda(x)V_{\lambda}(y)=
\prod_{i,j}\frac{1}{(1-x_{i}y_{j})^2}. $$
\end{proof}
%
By a standard argument, we have
\begin{corollary}
$$\langle W_\lambda(x),V_{\mu}(x)\rangle_{-1}=\delta_{\lambda \mu}.$$

\end{corollary}
%
\begin{theorem}
The matrix $A_n$ is integral.
\end{theorem}
\begin{proof}
We have
$$\sum_{\lambda \in P}
W_{\lambda}(x)V_{\lambda}(y)=
\prod_{i,j}\frac{1}{(1-x_{i}y_{j})^2}=\sum_{\lambda \in P}S_{\lambda}(x,x)S_{\lambda}(y).$$
Taking the inner product $\langle\ ,\ \rangle_{0}$ with $S_{\mu}(y)$, we obtain \begin{align*}
S_{\lambda}(x,x)&=\sum_{\mu \in P}\langle W_{\mu}(x)V_{\mu}(y), S_{\lambda}(y) \rangle_{0}\\
&=\sum_{\mu \in P}\langle V_{\mu}(y), S_{\lambda}(y) \rangle_{0} W_{\mu}(x).
\end{align*}
Thus we know
$$a_{\lambda \mu}=\langle V_{\mu}(y)
,S_{\lambda}(y)\rangle_{0}.$$
The numbers $ g_{\mu^r \nu}$ defined by $$P_{\mu^r}(y)=\sum_{\nu \in P} g_{\mu^r \nu}S_{\nu}(y)$$ are called the Stembridge coefficients and are known to be non-negative integers. Also one finds the following formula in \cite{cgr}. 
$$S_{\mu^{p}}(y^2)=\sum_{\xi \in P} \delta(\xi)c_{\xi[0], \xi[1]}^{\mu^d}S_{\xi}(y),$$
where $\delta(\xi)$ is the 2-$sign$ of $\xi$, $(\xi[0], \xi[1])$ is 
the 2-$quotient$ of $\xi$ (cf. \cite{ol}) and $c_{\xi[0], \xi[1]}^{\mu^d}$ is the Littlewood-Richardson coefficient. Hence
\begin{align*}
V_{\mu}(y)&=
P_{\mu^r}(y)S_{\mu^{p}}(y^2)\\
&=\sum_{\nu,\xi}\delta(\xi)
g_{\mu^r \nu}
c_{\xi[0], \xi[1]}^{\mu^d}
S_{\nu}(y)S_{\xi}(y)\\
&=\sum_{\lambda}\left(\sum_{\nu,\xi}\delta(\xi)g_{\mu^r \nu}
c_{\nu \xi}^{\lambda}
c_{\xi[0], \xi[1]}^{\mu^d} \right)
S_{\lambda}(y). 
\end{align*}
Therefore $$a_{\lambda \mu}=\sum_{\nu, \xi}\delta(\xi)g_{\mu^r \nu}
c_{\nu \xi}^{\lambda} 
c_{\xi[0], \xi[1]}^{\mu^d}$$ is an integer.
\end{proof}
%
\begin{example}

\[
\ \ \ \ \ A_{3}=
\bordermatrix{
  & (3,\emptyset) & (21,\emptyset) & (1,1) \cr
   (3) & 1 & 0 & 1\cr
  (21) & 1 & 1 & 0\cr
  (1^3)& 1 & 0 &-1}
  \]
  \[
  A_4=
\bordermatrix{
   & (4,\emptyset)& (31,\emptyset) & (\emptyset, 2) & (\emptyset,1^2)& (2,1)\cr
  (4) & 1 & 0 & 1 & 0 & 1\cr
  (31)& 1 & 1 & -1 & 0 & 1\cr
  (2^2) & 0 & 1 & 1 & 1 & 0\cr
  (1^4)& 1 & 0 & 0 & 1 & -1\cr
  (21^2) & 1 & 1 & 0 & -1 & -1}
\]

%
\end{example}

\ \\\\
\ \ \ As for the columns corresponding to $(\mu,\emptyset)$ with $\mu \in SP_n$, entries are non-negative integers. The submatrix consisting of these columns will be denoted by $\Gamma_n$. The entries of $\Gamma_n $ are the Stembridge coefficients, whose combinatorial nature has been known (\cite{st}, \cite{mac}).
\\\ \ \ Here we recall the definition of decomposition matrices for the $p$-modular representations of the symmetric group $S_n$.  Let $p$ be a fixed prime number.  A partition 
$\lambda=(\lambda_1,\lambda_2,\cdots,\lambda_{\ell})$ is said to be $p$-regular of there are no parts satisfying $\lambda_i=\lambda_{i+1}=\cdots=\lambda_{i+p-1} \geq 1.$ 
Note that a 2-regular partition is nothing but a strict partition. The set of $p$-regular partitions of $n$ is denoted by $P^{r(p)}_{n}.$
A partition $\rho=(1^{m_1}2^{m_2}\cdots)$ is said to be $p$-class regular if $m_{p}=m_{2p}=\cdots=0.$ Note that a 2-class regular partition is nothing but an odd partition. 
The set of $p$-class regular partitions of $n$ is denoted by $P^{c(p)}_{n}$. 
The $p$-Glaisher map $\lambda \mapsto \tilde{\lambda}$ is defined in a natural way.  
This gives a bijection between $P^{r(p)}_{n}$ and $P^{c(p)}_{n}$. 
For $\lambda \in P^{r(p)}_{n}$, we define the Brauer-Schur function $B_{\lambda}^{(p)}(t)$ indexed by 
$\lambda$ as follows.
$$B_{\lambda}^{(p)}(t)=\sum_{\rho \in P^{c(p)}_{n}}\varphi_{\rho}^{\lambda}\frac{t_1^{m_1}t_2^{m_2}\cdots}{m_1!m_2!\cdots} \quad \in \SF(0)_{n},$$
where $\varphi_{\rho}^{\lambda}$ is the irreducible Brauer character corresponding to $\lambda$, evaluated at the $p$-regular conjugacy class $\rho$. 
These functions form a basis for the space
$\SF^{(p)}_{n}=\SF^{(p)} \cap \SF(0)_{n}$, where 
$$\SF^{(p)}= {\mathbb C}[t_j; j \geq 1, j \not \equiv 0 \, ({\rm{mod}}\, p) \}.$$

Given a Schur function $S_{\lambda}(t)$, define the $p$-reduced Schur function $S_{\lambda}^{(p)}(t)$ by "killing" all variables $t_p, t_{2p}, \cdots $;
$$S_{\lambda}^{(p)}(t)=S_{\lambda}(t)|_{t_{jp}=0}.$$ 
These $p$-reduced Schur functions are no longer linearly independent. All linear relations among these polynomials are known (cf. \cite{any}).
The $p$-decomposition matrix $D_{n}^{(p)}=(d_{\lambda \mu})$ is defined by 
$$S_{\lambda}^{(p)}(t)=\sum_{\mu \in P^{r(p)}_{n}}d_{\lambda \mu}B_{\mu}^{(p)}(t)$$ for $\lambda \in P_{n}$, and are known to satisfy the properties; $d_{\lambda \mu} \in {\mathbb Z}_{\geq 0},$ $d_{\lambda \mu}=0$ unless $\mu \geq \lambda$ and $d_{\lambda \lambda}=1.$ Here $``\geq "$ denotes the dominance order. 
\\\ \ \ Now let us go back to the case of $p=2$. By definition, the Stembridge coefficients $\gamma_{\lambda \mu} \quad (\lambda \in P_{n},\mu \in SP_{n})$ appear as 
$$S_{\lambda}^{(2)}(t)=\sum_{\mu \in SP_{n}}\gamma_{\lambda \mu}Q_{\mu}(t).$$
Looking at the matrices $D_n^{(2)}=(d_{\lambda \mu})$ and $\Gamma_n=(\gamma_{\lambda \mu})$, one observes that they are 
``very similar". We consider the Cartan matrix $C_n^{(2)}={}^tD_n^{(2)}D_n^{(2)}$ and the correspondent 
$G_n={}^t{\Gamma_n}{\Gamma_n}.$ There is a compact formula for the elementary divisors of $C_n^{(2)}$(\cite{unya}):\ $2^{\ell(\tilde{\lambda})-\ell(\lambda)}$ for $\lambda \in SP_{n}.$

\begin{theorem}

The elementary divisors of $C_n^{(2)}$ and $G_n$ coincide.

\end{theorem}
\begin{proof}\
We put $\tilde{Z}_n=(2^{\frac{\ell(\lambda)-\ell(\rho)+\epsilon}{2}}\zeta^{\lambda}_{\rho})_{\lambda \in SP_n, \rho \in OP_n}$, $\Phi_{n}^{(2)}=(\varphi^{\lambda}_{\rho})_{\lambda \in SP_n, \rho \in OP_n}$ and $X_{n}^{(2)}=(\chi^{\lambda}_{\rho})_{\lambda \in P_n,\rho \in OP_n}$. The transition matrix $T_n=(t_{\lambda \mu})_{\lambda,\mu \in SP_n}$ is defined by $$B_{\lambda}^{(2)}(t)=\sum_{\mu \in SP(n)}t_{\lambda \mu}Q_{\mu}(t).$$\ \ \ By definition of $\Gamma_n$ and $D^{(2)}_{n},$
\ we have $X_{n}^{(2)}=\Gamma_n \tilde{Z}_n=D_n^{(2)}\Phi^{(2)}_n$ and $\Phi^{(2)}_{n}=T_n\tilde{Z}_n$. Hence, we have
$\Gamma_n=X^{(2)}_{n}\tilde{Z}^{-1}_{n}=D_n^{(2)}\Phi^{(2)}_n\tilde{Z}^{-1}_n=D_n^{(2)}T_n.$ 
\\\ \ \ The matrix $\Gamma_n$ has the following properties; $\gamma_{\lambda \mu} \in {\mathbb Z}_{\geq 0}$, $\gamma_{\lambda \mu}=0$ unless $\mu \ge \lambda$ and $\gamma_{\lambda \lambda}=1$ (\cite{st}). Fix a total order in the set of partitions which is compatible with the dominance order, and
we shall write $(d_{ij}),(\gamma_{ij})$ and $(t_{ij})$ in place of $(d_{\lambda \mu}),(\gamma_{\lambda \mu})$ and $(t_{\lambda \mu})$, respectively. Looking at the first row of $D_n^{(2)}T_n$, we have $$\delta_{1j}=\gamma_{1j}=\sum_{k=1}^{}d_{1k}t_{kj}.$$ This shows that $t_{1j}=\delta_{1j}.$ As for the second row of $D_{n}^{(2)}T_n$, we have $$\delta_{2j}=\gamma_{2j}=\sum_{k=1}^{}d_{2k}t_{kj}=\sum_{k=2}^{}d_{2k}t_{kj}\ \ (j \geq 2).$$ This shows that $t_{2j}=\delta_{2j}.$ Inductively, we can see that $T_n$ is a lower unitriangular integral matrix. Therfore the matrix $D_{n}^{(2)} $ and $\Gamma_n$ are transformed to each other by column operations. By a standerd argument we see that the elementary divisors of $C_n^{(2)}$ and $G_n$ coincide.


\end{proof}

Our transition matrix $A_n=(a_{\lambda \mu})_{\lambda,\mu \in P_{n}}$ can be regarded as a common extension of the matrix $\Gamma_n$ of  Stembridge coefficients and 
the decomposition matrix $D_n^{(2)}$. 

%

\begin{theorem}
$$|\det A_{n}|=2^{k_{n}},$$
where $k_{n}=\sum_{\lambda \in P_{n}}\ell(\lambda^e)
=\sum_{\lambda \in P_{n}}(\ell(\tilde{\lambda^{r}})-\ell(\lambda^{r}))$.
\end{theorem}
%
\begin{proof}
We have four bases of $\SF(0)_{n}$;
 $S=(S_{\lambda}(x))_{\lambda \in P_{n}}$, $\tilde{S}=(S_{\lambda}(x,x))_{\lambda \in P_{n}}$, $V=(P_{\lambda^{r}}(x)S_{\lambda^{d}}(x^2))_{\lambda\in P_{n}}$ and $W=(Q_{\lambda^{r}}(x)S_{\lambda^{d}}(x^2))_{\lambda\in P_{n}}$. 
From Corollary 4.2, $W$ and $V$ are dual to each other with respect to the inner product $\langle\ ,\ \rangle_{-1}$. Likewise, $\tilde{S}$ and $S$ are dual to each other.
Hence we obtain
$${}^tM(S,V)_{n}M(\tilde{S},W)_{n}=I,$$ where $M(S,V)_{n}$ denotes the transition matrix from the basis $S$ to the basis $V$ for $\SF(0)_{n}.$ Since
$$M(S,V)_{n}=M(S,\tilde{S})_{n}A_{n}M(W,V)_{n},$$
we see that $$(\det A_n)^2=\frac{1}{\det M(S,\tilde{S})_{n}\det M(W,V)_{n}}.$$
Let $X_n=(\chi^{\lambda}_{\rho})_{\lambda \rho}$
be the character table of
${\mathfrak S}_{n}$. 
We put $R_{n}={\rm diag}(z_{\rho};\rho \in P_n)$ and $L_{n}={\rm diag}(2^{\ell(\rho)};\rho \in P_n).$
Then we see that
\begin{align*}
\det M(S,\tilde{S})_n&=\det M(S,p)_{n}\det M(p,\tilde{S})_{n}\\
&=\det X_nR_{n}^{-1}
\det L_{n}^{-1}{}^tX_n\\
&=\det L_{n}^{-1},
\end{align*}
and
\begin{align*}
\det M(W,V)_{n}&=\prod_{\lambda \in P_{n}}2^{\ell(\lambda^{r})}.
\end{align*}
Hence we have $$\det A_{n}^2=\prod_{\lambda \in P_{n}}2^{\ell(\lambda)-\ell(\lambda^{r})}
=\prod_{\lambda \in P_{n}}2^{2\ell(\lambda^{d})}=\prod_{\lambda \in P_{n}}2^{2\ell(\lambda^e)}.$$
\end{proof}
%
Here is a small list of $k_n.$
$$
\begin{array}{c|ccccccccc}
n&1&2&3&4&5&6&7&8&\cdots\\
\hline
k_{n}&0&1&1&4&5&11&15&28&\cdots
\end{array}.$$

%
 


Next we consider the ``Cartan-like" matrix $^{t}A_{n}A_{n}.$
The Frobenius formula for $W_{\lambda}$ reads 
$$p_{\sigma}p_{2\rho}=\sum_{\lambda \in P_{n_{0},n_{1}}}2^{-\ell(\lambda^{r})}
 X_{\sigma}^{\lambda^{r}}\chi_{\rho}^{\lambda^{d}}W_{\lambda}(x) $$ for $\sigma \in OP_{n_{0}}$ and
  $\rho \in P_{n_{1}}$, where the Green function $X^{\lambda}_{\sigma}$ is defined by
\begin{align*}
 Q_{\lambda}(x)&=\sum_{\sigma}2^{\ell(\sigma)}z_{\sigma}^{-1}X_{\sigma}^{\lambda}p_{\sigma}
\end{align*} for $\lambda \in SP_n.$
This formula shows that the transition matrix $M(p,W)_{n}$ is, after a suitable sorting of rows and columns, decomposed into diagonal blocks, each block indexed by the pair $(n_{0},n_{1})$ with $n_{0}+2n_{1}=n$.

We have
 \begin{align*}
 {}^tA_{n}A_{n}&={}^tM(p,W)_{n}{}^tM(\tilde{S},p)_{n}M(\tilde{S},p)_{n}M(p,W)_{n}\\
&={}^tM(p,W)_{n}({}^tL_nX_{n}^{-1})(X_nR_{n}^{-1}L_n)M(p,W)_{n}\\
&={}^tM(p,W)_{n} {}^tL_{n}R_{n}^{-1}L_n M(p,W)_{n}.
 \end{align*}
Since $L_{n}^{2}R_{n}^{-1}$ is diagonal matrix, ${}^tA_{n}A_{n}$ is block diagonal matrix, each block indexed by the pair $(n_{0},n_{1}).$ Let denote $B_{n_{0},n_{1}}$ the corresponding block in ${}^tA_nA_n.$ Note that the ``principal" block $B_{n,0}$ is nothing but the matrix $G_n$.
\begin{example}
\[
\ \ \ \ \ {}^{t}A_{3}A_3=
\bordermatrix{
  & (3,\emptyset) & (21,\emptyset) & (1,1) \cr
   (3,\emptyset) & 3 & 1 & 0\cr
  (21,\emptyset) & 1 & 1 & 0\cr
  (1,1)& 0 & 0 &2},
  \]
  \[
  {}^{t}A_{4}A_4=
\bordermatrix{
   & (4,\emptyset)& (31,\emptyset) & (\emptyset, 2) & (\emptyset,1^2)& (2,1)\cr
  (4,\emptyset) & 4 & 2 & 0 & 0 & 0\cr
  (31,\emptyset)& 2 & 3 & 0 & 0 & 0\cr
  (\emptyset,2) & 0 & 0 & 3 & 1 & 0\cr
  (\emptyset,1^2)& 0 & 0 & 1 & 3 & 0\cr
  (2,1) & 0 & 0 & 0 & 0 & 4}.
\] 
\end{example}

%
\begin{theorem}
 \begin{align*}
|\det B_{n_{0},n_{1}}|
&=2^{ \sum_{\lambda \in P_{n_{0},n_{1}}}(\ell(\tilde{\lambda^{r}})+\ell(\lambda^{d})-\ell(\lambda^{r}))}. 
\end{align*}
\end{theorem}
%
\begin{proof}
We have
\begin{align*}
^{t}A_{n}A_{n}&={}^tM(p,W)_{n} {}^tL_{n}R_{n}^{-1}L_n M(p,W)_{n}\\
&={}^tM(p,W)_{n} {}^tL_{n}R_{n}^{-1}L_n M(p,\tilde{p})_{n}M(\tilde{p},V)_{n}M(V,W)_{n}\\
&={}^tM(p,W)_{n} {}^tL_{n}R_{n}^{-1}M(\tilde{p},V)_{n}M(V,W)_{n},
\end{align*}
where $\tilde{p}_{\rho}(x)=p_{\rho}(x,x).$
Note that $^{t}L_{n}$, $R_{n}^{-1}$, $M(V,W)_{n}$ are diagonal matrices. By requiring that $\langle p_{\rho},p_{\sigma}\rangle_{-1}=2^{-\ell(\rho)}z_{\rho}\delta_{\rho \sigma}$, we obtain $\langle p_{\rho},\tilde{p_{\sigma}}\rangle_{-1}=2^{\ell(\sigma)-\ell(\rho)}z_{\rho}\delta_{\rho \sigma}$. Hence, $$\det(M(\tilde{p},V)_{n}\ ^{t}M(p,W)_{n}Z_{n}^{-1})=\det I.$$
We recall that $^{t}A_{n}A_{n}$ is block diagonal. We have  \begin{align*}
|\det B_{n_{0},n_{1}}|
&=2^{ \sum_{\lambda \in P_{n_{0},n_{1}}}(\ell(\tilde{\lambda^{r}})+\ell(\lambda^{d})-\ell(\lambda^{r}))}. 
\end{align*}
\end{proof}

For the principal block $B_{n,0}$, we have 
$$|\det B_{n,0}|=2^{\sum_{\lambda \in SP_n}(\ell(\tilde{\lambda})-\ell(\lambda))}.$$
\ \ \ We conclude this note with an inner product expression of ${}^tA_{n}A_{n}$. 
\begin{proposition}
$${}^tA_{n}A_{n}=\bigg(
\langle P_{\lambda^{r}}(x),P_{\mu^r}(x) \rangle_{0}\langle S_{\lambda^{d}}(x^2),S_{\mu^d}(x^2)\rangle_{0}
\bigg)_{\lambda,\mu}.
$$
\end{proposition}
\begin{proof}
We have already given
 \begin{align*}
{}^tA_{n}A_{n}&=
 {}^tM(p,W)_{n}L_{n}^{2}R_{n}^{-1}M(p,W)_{n}.
 \end{align*}
 Hence
\begin{align*}
\sum_{\sigma,\rho}&2^{-\ell(\lambda^{r})-\ell(\mu^r)}
X_{\sigma}^{\lambda^{r}}X_{\sigma}^{\mu^r}
\chi_{\rho}^{\lambda^{d}}\chi_{\rho}^{\mu^d}
2^{2\ell(\sigma)+2\ell(\rho)}
z_{\sigma}^{-1}
z_{2\rho}^{-1}\\
&=2^{-\ell(\lambda^{r})-\ell(\mu^r)}
\sum_{\sigma,\rho}
(2^{2\ell(\sigma)}X_{\sigma}^{\lambda^{r}}X_{\sigma}^{\mu^r}z_{\sigma}^{-1})
(2^{\ell(\rho)}\chi_{\rho}^{\lambda^{d}}\chi_{\rho}^{\mu^d}z_{\rho}^{-1})\\
&=2^{-\ell(\lambda^{r})-\ell(\mu^r)}
\langle Q_{\lambda^{r}},Q_{\mu^r} \rangle_{0}
\langle S_{\lambda^{d}}(x^2),S_{\mu^d}(x^2)\rangle_{0}.
 \end{align*}
\end{proof}

\end{document}